\documentclass{amsart}
\usepackage{amsmath, amssymb, amsthm}
\usepackage{graphicx}
\usepackage{mathrsfs}
\usepackage{tikz}
\usepackage{fancyhdr}

\usetikzlibrary{arrows,chains,matrix,positioning,scopes}

\makeatletter
\tikzset{join/.code=\tikzset{after node path={%
\ifx\tikzchainprevious\pgfutil@empty\else(\tikzchainprevious)%
edge[every join]#1(\tikzchaincurrent)\fi}}}
\makeatother
\tikzset{>=stealth',every on chain/.append style={join},
         every join/.style={->}}
\tikzstyle{labeled}=[execute at begin node=$\scriptstyle,
   execute at end node=$]

\newtheorem{theorem}{Theorem}[section]
\newtheorem{proposition}{Proposition}[section]
\newtheorem{lemma}{Lemma}[section]
\newtheorem{definition}{Definition}[section]
\newtheorem{corollary}{Corollary}[section]

\title{}

\begin{document}

\maketitle

\begin{center}
{\bf TROPICAL EMBEDDINGS OF METRIC GRAPHS}\\[.5cm]
{\sc Sylvain Carpentier, Adan Medrano Martin del Campo}\\[.5cm]
{July 29, 2015}
\end{center}
\begin{abstract}
Every graph $\Gamma$ can be embedded in the plane with a minimal number of edge intersections, called its classical crossing number $\text{cross}\left(\Gamma\right)$. In this paper, we prove that if $\Gamma$ is a metric graph it can be realized as a tropical curve in the plane with exactly $\text{cross}\left(\Gamma\right)$ crossings, where the tropical curve is equipped with the lattice length metric. Our result has an application in algebraic geometry, as it enables us to construct a rational map of non-Archimedean curves into the projective plane, whose tropicalization is almost faithful when restricted to their skeleton.
\end{abstract}

\tableofcontents

\newpage

\section{Introduction}

Given a graph $\Gamma$ we denote by $\text{cross}\left(\Gamma\right)$ its classical crossing number, which is the minimal number of edge intersections that occur when $\Gamma$ is embedded into $\mathbb{R}^{2}$. Rabinoff and Baker used the fact that any graph $\Gamma$ can be embedded in $\mathbb{R}^{3}$ to prove the following result.\\

\begin{theorem}
{\bf (Baker, Rabinoff)} \cite[p. 19, Theorem 8.2]{BakerRabinoff13} If $X$ is a smooth proper $K-curve$ and $\Gamma$ is any skeleton of $X$, then there is a rational map $f:X\dashrightarrow \mathbb{P}^{3}$ such that the restriction of its tropicalization to $\Gamma$ is an isometry onto its image.\\
\end{theorem}

An analogous result cannot be obtained for a rational map $f:X\dashrightarrow \mathbb{P}^{2}$, since $\text{trop}\left(f\right)\left(\Gamma\right)\subset \mathbb{R}^{2}$ might have a crossing number greater than $0$, impeding the possibility of a faithful tropicalization. However, in this paper, we give the following adaptation of the previous result in the two dimensional case.\\

\begin{theorem}
Let $X$ be a $K-curve$ and let $\Gamma$ be a skeleton of $X$. There exists a rational map $f:X\dashrightarrow \mathbb{P}^{2}$ such that its tropicalization $\text{trop}\left(f\right)=-\log\left|f\right|$ restricted to $\Gamma$ is an isometry up to $\text{cross}\left(\Gamma\right)$ crossings.\\
\end{theorem}

This is in some sense the most faithful tropicalization we can obtain, since no less than $\text{cross}\left(\Gamma\right)$ crossings may appear in an immersion of $\Gamma$ into the plane. To obtain to this result, we first prove the following theorem.\\

\begin{theorem}
Every abstract metric graph $\Gamma$ admits an isometric balanced embedding up to $\text{cross}\left(\Gamma\right)$ crossings restricted to the embedded edges of $\Gamma$.\\
\end{theorem}

We will define isometric balanced embeddings in section $2.3$. In section $1$, we present the construction of the isometric balanced embedding of a metric graph $\Gamma$, and in section $2$ we use the lemma by Baker and Rabinoff along with our embedded construction in order to prove our main result.\\

\section{Isometric Balanced Embedding with $\text{cross}\left(\Gamma\right)$ Crossings}

In this section we present a construction of an isometric balanced embedding of a given metric graph, which has exactly $\text{cross}\left(\Gamma\right)$ crossings. We begin by recalling the definitions of metric and balanced graphs.\\

\subsection{Balanced Graphs}

:\\
\begin{definition}
A {\bf $1-$dimensional polyhedral complex} is a connected, finite union of segments and rays in the plane, where the transversal intersection of any two segments or rays is called a {\bf vertex}, and any segments and rays which endpoints are vertices are called {\bf edges}. We say the polyhedral complex is {\bf rational} if all its edges have rational slopes. \\ 
\end{definition}

\begin{definition}
Let $p$ be a vertex and let $r$ be an edge of the polyhedral complex with rational slope going from $p$.  The  {\bf primitive vector} of $r$ is the vector $\vec{r}\left(p\right)=\left(m, n\right)^{t}$ such that $m$ and $n$ are relatively prime integers and $\vec{r}\left(p\right)$ is parallel to $r$ and pointing out of the vertex $p$.
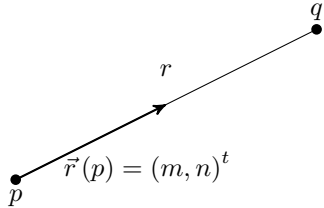
\begin{figure}[h]
\begin{tikzpicture}
\fill (0, 0) circle (2pt) node[below] {$p$};
\draw (0, 0) -- (4, 2);
\draw[thick, ->] (0, 0) -- (2, 1);
\draw (1.75, .5) node[right, below] {$\vec{r}\left(p\right)=\left(m, n\right)^{t}$};
\draw (2, 1.25) node[left, above] {$r$};
\fill (4, 2) circle (2pt) node[above] {$q$};
\end{tikzpicture}
\caption{A line segment and its primitive vector $\vec{r}\left(p\right)$.}
\end{figure}

\end{definition}

\begin{definition}
A {\bf weighted $1$-dimensional polyhedral complex} is a polyhedral complex for which every edge $e$ is assigned a positive integer $w_e$ called its {\bf weight}.
Let $p$ be a vertex of a weighted $1$-dimensional polyhedral complex of degree $n$. The $n$ edges $r_{1}, r_{2}, \ldots, r_{n}$ departing from $p$ are assigned weights $\omega_{1}, \omega_{2}, \ldots, \omega_{n}$, respectively, and have rational slopes. We say that the {\bf balancing condition} holds on $p$ if 
\\
\[\sum_{i=1}^{n}{\omega_{i}\vec{r}_{i}\left(p\right)}=0.\]
\\
where $\vec{r}_{i}\left(p\right)$ is the primitive vector of the edge $r_{i}$.\\
\end{definition}

\begin{definition}
A {\bf balanced graph} is a weighted $1-$dimensional rational polyhedral complex for which the balancing condition holds on each vertex.\\
\end{definition}

\subsection{Tropical Length}

:\\
\\
We shall restrict ourselves to connected balanced graphs. We will now see that one can endow a balanced graph with a metric, as follows.\\

\begin{definition}
Let $\ell$ be a segment with rational slope in $\mathbb{R}^{2}$ with endpoints $p_{1}$ and $p_{2}$. Let $\mathbf{v}_{l}$ be the primitive vector of $l$  going out from $p_{1}$. Let $\alpha$ be the (unique) positive real number such that
\\
\[\alpha \mathbf{v}_{l}=p_{2}-p_{1}.\]
\\
$\alpha$ is called {\bf tropical length} of $\ell$.\\
\end{definition}

Bounded edges of our balanced graph can be assigned the above tropical length, and rays are assigned infinite lengths (but finite portion of rays have a finite length associated to them). This endows the balanced graph with a metric graph structure.\\

\begin{definition}
An {\bf abstract metric graph $\Gamma$} is a finite graph such that every edge is assigned a positive real number or infinity, which is called the {\bf length} of the edge. Any edge with infinite length must have a degree $1$ vertex as one its endpoints, which will be referred to as an {\bf infinite vertex}.\\
\end{definition}

\begin{definition}
A {\bf subdivision of $\Gamma$} consists in replacing an edge with two consecutive edges whose lengths add up to the length of the original edge. If the original edge was infinite, then the subdivided edge incident to the infinite vertex must also be infinite.\\
\\
A {\bf reverse subdivision} consists in removing a vertex $p$ of degree $2$ which is connected to different vertices, and replacing the edges of $p$ by a single edge connecting the vertices to which $p$ was connected. The length of this new edge is defined as the sum of the two original edges. If one of the removed edges had infinite length, then the new edge also has infinite length. \\
\\
An {\bf elementary tropical modification of $\Gamma$} is the addition of an infinite edge to a non-infinite vertex.\\
\end{definition}

\begin{definition}
Let $\Gamma$ be a metric graph. A {\bf tropical modification} is a finite sequence of elementary tropical modifications, subdivisions, and reverse subdivisions. \cite[p. 4]{CDMY14}\\
\end{definition}

\subsection{Balanced Embeddings}

:\\
\begin{definition}
Let $\Gamma$ be an abstract metric graph. An {\bf isometric balanced embedding} of $\Gamma$ is a realization of a tropical modification $\Gamma'$ of $\Gamma$, as a balanced graph $B\left(\Gamma\right)$ with weights $1$ on the bounded edges of $\Gamma$, and which preserves the lengths, restricted to $\Gamma$.\\ 
\end{definition}

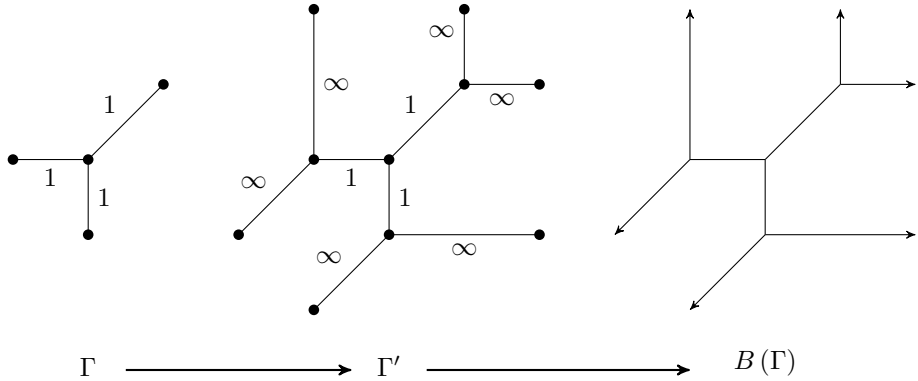
\begin{figure}[h]
\centering
\begin{tikzpicture}
\draw (6, 1) -- (5, 0);
\draw[->] (6, 1) -- (6, 2);
\draw[->] (6, 1) -- (7, 1);
\draw (4, 0) -- (5, 0);
\draw[->] (4, 0) -- (4, 2);
\draw[->] (4, 0) -- (3, -1);
\draw (5, -1) -- (5, 0);
\draw[->] (5, -1) -- (7, -1);
\draw[->] (5, -1) -- (4, -2);

\draw (1, 1) -- (0, 0);
\draw (1, 1) -- (1, 2);
\draw (1, 1) -- (2, 1);
\draw (-1, 0) -- (0, 0);
\draw (-1, 0) -- (-1, 2);
\draw (-1, 0) -- (-2, -1);
\draw (0, -1) -- (0, 0);
\draw (0, -1) -- (2, -1);
\draw (0, -1) -- (-1, -2);

\fill (0, 0) circle (2pt);
\fill (1, 1) circle (2pt); 
\fill (-1, 0) circle (2pt); 
\fill (0, -1) circle (2pt);
\fill (1, 2) circle (2pt);  
\fill (2, 1) circle (2pt); 
\fill (2, -1) circle (2pt); 
\fill (-1, -2) circle (2pt);
\fill (-1, 2) circle (2pt); 
\fill (-2, -1) circle (2pt);

\draw (.5, .5) node[above left] {$1$};
\draw (-.5, 0) node[below] {$1$};
\draw (0, -.5) node[right] {$1$};
\draw (-1, 1) node[right] {$\infty$};
\draw (-1.5, -.5) node[above left] {$\infty$};
\draw (1, -1) node[below] {$\infty$};
\draw (-.5, -1.5) node[above left] {$\infty$};
\draw (1.5, 1) node[below] {$\infty$};
\draw (1, 1.5) node[above left] {$\infty$};

\draw (6, 1)+(-9, 0) -- (-4, 0);
\draw (4, 0)+(-9, 0) -- (-4, 0);
\draw (5, -1)+(-9, 0) -- (-4, 0);

\fill (5, 0)+(-9, 0) circle (2pt);
\fill (6, 1)+(-9, 0) circle (2pt); 
\fill (4, 0)+(-9, 0) circle (2pt); 
\fill (5, -1)+(-9, 0) circle (2pt);

\draw (-3.5, .5) node[above left] {$1$};
\draw (-4.5, 0) node[below] {$1$};
\draw (-4, -.5) node[right] {$1$};

\draw[thick, ->] (-3.5, -2.8) -- (-.5, -2.8);
\draw[thick, ->] (.5, -2.8) -- (4, -2.8);

\draw (-4, -3) node[above] {$\Gamma$};
\draw (0, -3) node[above] {$\Gamma'$};
\draw (5, -3) node[above] {$B\left(\Gamma\right)$};
\end{tikzpicture}
\caption{A metric graph $\Gamma$, one of its tropical modifications $\Gamma'$ and a realization $B\left(\Gamma\right)$ of $\Gamma$ as a balanced graph.}
\end{figure}

It may be the case that any such realization of a metric graph $\Gamma$ has intersections, since the classical crossing number of $\Gamma$, $\text{cross}\left(\Gamma\right)$, might be greater than $0$. We will show that it is possible to have an isometric balanced embedding of any metric graph $\Gamma$, with no more than  $\text{cross}\left(\Gamma\right)$ crossings restricted to the embedded edges of $\Gamma$.\\

\subsection{Linear Embedding of $\Gamma$ with $\text{cross}\left(\Gamma\right)$ Crossings}
:\\
\\
We present a construction of an embedding of an abstract metric graph $\Gamma$ as a rational polyhedral complex with $\text{cross}\left(\Gamma\right)$ crossings. In this section, we don't care about isometry, and this issue will be dealt with in the next sections.\\

\begin{lemma}
Let $\Gamma$ be a metric graph. There exists an embedding of $\Gamma$ in $\mathbb{R}^{2}$ with exactly $\text{cross}\left(\Gamma\right)$ self crossings of edges of $\Gamma$ such that every embedded edge is piecewise linear with rational slopes.
\end{lemma}

\begin{proof}
First, we make an appropriate tropical modification on $\Gamma$ to obtain a graph $\Gamma'$ such that $\Gamma'$ is a simple graph (with no loops or multiple edges). This doesn't increase the number of crossings, so $\text{cross}\left(\Gamma\right)=\text{cross}\left(\Gamma'\right)$. We can embed $\Gamma'$ in the plane with $\text{cross}\left(\Gamma\right)$ crossings by definition. We can make suitable subdivisions to $\Gamma'$, obtaining a graph $\Gamma''$, in order to locate the vertices of $\Gamma''$ on the plane, satisfying that every segment is a straight line and we still have $\text{cross}\left(\Gamma\right)$ crossings. We do this as follows. Let $G$ be an embedding of $\Gamma'$ in $\mathbb{R}^{2}$ such that $G$ has exactly $\text{cross}\left(\Gamma\right)$ crossings. Now, let $\varepsilon>0$ such that the $\varepsilon$ neighborhoods $N_{\varepsilon}\left(v\right)$ around each vertex and every crossing of $G$ are disjoint. Let $V\left(G\right)$ and $C\left(G\right)$ be the set of vertices and crossings of $G$, respectively. Then, consider
\\
\[G'=G-\left(\left(\bigcup_{v\in V\left(G\right)}N_{\varepsilon}\left(v\right)\right)\cup\left(\bigcup_{c\in C\left(G\right)}{N_{\varepsilon}\left(c\right)}\right)\right).\]
\\
We have that $G'$ contains finitely many disjoint connected bounded closed paths. Let $\text{Comp}\left(G'\right)$ be the set of connected components of $G'$. For any $\mathcal{C}, \mathcal{C'}\in \text{comp}\left(G'\right)$ there exists some $\delta_{\mathcal{C}, \mathcal{C}'}>0$ such that the $\delta_{\mathcal{C}, \mathcal{C}'}$ neighborhoods of any point in $\mathcal{C}$ and any point in $\mathcal{C}'$ are disjoint. Hence, let
\\
\[\delta=\min_{\mathcal{C}, \mathcal{C}'\in \text{Comp}\left(G'\right)}\delta_{\mathcal{C}, \mathcal{C}'}.\]
\\
We must have that these component neighborhoods,
\\
\[\mathcal{N}_{\mathcal{C}}=\bigcup_{c\in \mathcal{C}}{N_{\delta}\left(c\right)}\]
\\
are disjoint. Suppose $\mathcal{C}\in \text{Comp}\left(G'\right)$ has euclidean length $L$. Subdivide $\mathcal{C}$ into $N$ paths, each of length $\frac{L}{N}<\delta$. Joining these new points with straight lines ensures that each of these segments is contained in $\mathcal{N}_{\mathcal{C}}$. Doing this for every $\mathcal{C}'\in \text{Comp}\left(G'\right)$, we have constructed piecewise linear non intersecting paths. Finally, we join the endpoints of each connected component $\mathcal{C}$ the following way. If one of the two endpoints of $\mathcal{C}$ was originally connected to a vertex of $G$, we join them back with a straight line. If one of the endpoints of $\mathcal{C}$ was originally connected to an intersection points $I$, then we do the following.

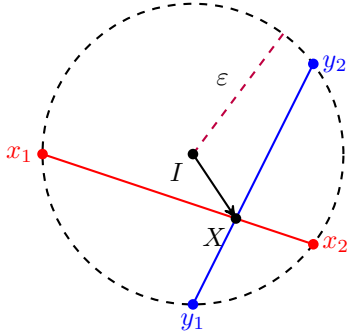
\begin{figure}[h]
\centering
\begin{tikzpicture}
\draw[thick, dashed] (0, 0) circle (2cm);
\draw[thick, dashed, purple] (0, 0) -- (1.2, 1.6);
\draw (.6, .8) node[above left] {$\varepsilon$}; 
\fill[red] (-2, 0) circle (2pt) node[left] {$x_{1}$};
\fill[red] (1.6, -1.2) circle (2pt) node[right] {$x_{2}$};
\fill[blue] (0, -2) circle (2pt) node[below] {$y_{1}$};
\fill[blue] (1.6, 1.2) circle (2pt) node[right] {$y_{2}$};
\fill (0, 0) circle (2pt) node[below left] {$I$};
\draw[thick, red] (-2, 0) -- (1.6, -1.2);
\draw[thick, blue] (0, -2) -- (1.6, 1.2);
\fill (intersection cs: first line={(-2, 0)--(1.6, -1.2)}, second line={(0, -2)--(1.6, 1.2)}) coordinate (X) circle (2pt) node[below left] {$X$};
\draw[thick, ->] (0, 0) -- (X); 
\end{tikzpicture}
\caption{Rejoining endpoints of intersecting paths.}
\end{figure}

We consider the four connected components which were connected to $I$, and let their endpoints that were connected to $I$ be $x_{1}, x_{2}, y_{1}, y_{2}$, where $x_{1}$ and $x_{2}$ were originally part of the same edge, and $y_{1}$ and $y_{2}$ were part of the same edge. Recall these four points lie on a circle. If the segments $x_{1}x_{2}$ and $y_{1}y_{2}$ do not intersect, one could originally change the path over the edge joining $x_{1}$ and $x_{2}$ for the segment joining them, and do the same with $y_{1}$ and $y_{2}$, and then our intersection $I$ vanishes, reducing the number of crossings of $G$, which is a contradiction. Hence, $x_{1}x_{2}$ and $y_{1}y_{2}$ cross. Therefore, joining these four endpoints by drawing the straight lines $x_{1}x_{2}$ and $y_{1}y_{2}$ creates our intersection again. This creates an embedding of $\Gamma$ with straight lines, which has $\text{cross}\left(\Gamma\right)$ crossings.\\
\\ 
Now, consider a vertex $v\in \mathbb{R}^{2}$. Since $\mathbb{Q}$ is dense in $\mathbb{R}$ then we can consider any neighborhood of radius $\varepsilon>0$ of $v$, $N_{\varepsilon}\left(v\right)$ and find $w_{v}\in \mathbb{Q}^{2}$ such that $w\in N_{\varepsilon}\left(v\right)$. Note that we can take disjoint neighborhoods around every vertex of $\Gamma$, and we move every vertex $v$ to a point $w_{v}\in \mathbb{Q}^{2}$ as described before. Choosing $\varepsilon$ to be sufficiently small, we can ensure that the number of self crossings does not change. Finally, since for every $v\in \Gamma$ we have $w_{v}\in \mathbb{Q}^{2}$, then the slope of any segment joining the points $w_{v}$ and $w_{u}$ with $u, v\in \Gamma$ must be rational. In addition, the crossing of the lines can be modified so that the lines which cross have as primitive vectors the vectors $\left(1, 0\right)^{t}$ and $\left(0, 1\right)^{t}$ as in figure 4, which which preserves having rational slopes.\\

\begin{figure}[h]
\centering
 \begin{tikzpicture}
 \draw[thick] (0, 3) -- (6, 0);
 \draw[thick] (1.5, 0) -- (4.5, 3);
 \draw[->] (5.5, 1.5) -- (6.5, 1.5);
 \draw[thick] (6, 3) -- (8, 2) -- (8, 1) -- (10, 1) -- (12, 0);
 \draw[thick] (7.5, 0) -- (8, .5) -- (9.5, .5) -- (9.5, 2) -- (10.5, 3);
 \draw[thick, blue, dashed] (8, 2) -- (10, 1);
 \draw[thick, blue, dashed] (8, .5) -- (9.5, 2);
 \fill[red] (9, 1.5) circle (2pt) node[above] {$X$};
 \fill[red] (3, 1.5) circle (2pt) node[above] {$X$};
 \fill[red] (9.5, 1) circle (2pt) node[below right] {$X'$};
 \draw[thick, ->] (9, 1.5) -- (9.5, 1);
 \end{tikzpicture}
 \caption{Transformation of intersections into orthogonal intersections of paths.}
 \end{figure}
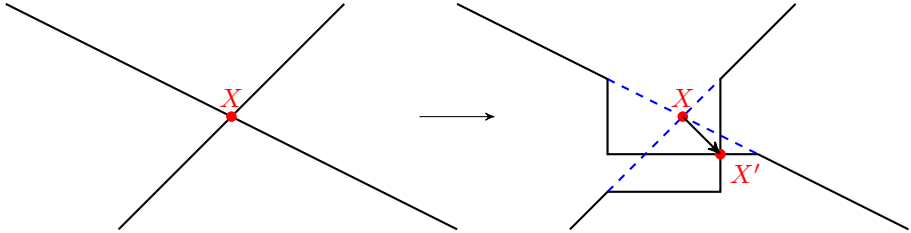
\end{proof}

\subsection{Cr\'eneaux}

:\\
\\
Now that we have embedded $\Gamma$ in the plane, we introduce our main tool to modify the tropical edge length of an edge.\\

\begin{definition}
Consider a segment $\ell$ in $\mathbb{R}^{2}$. A {\bf cr\'eneau on $\ell$} is the insertion of a finite piecewise linear path on $\ell$, along with some rays such that
\\
\begin{itemize}
\item{Rays are attached to every vertex of the cr\'eneau, along with a suitable weight on them, make the balancing condition holds on these vertices.}
\item{All angles formed by this path are $\pm\frac{\pi}{2}$.}
\item{The sequence of the signs of the angles formed by the path is of the form
\\
\[+ - - + + - - + \cdots - + + - \]}
\item{All segments parallel to $\ell$ have (euclidean) length $\alpha$ and all segments perpendicular to $\ell$ have length $\beta$.}
\\
\\
\end{itemize}

\begin{figure}[h]
\begin{tikzpicture}
\draw[very thick] (0, 0) -- (3, 0) -- (3, 1) -- (3.5, 1) -- (3.5, 0) -- (4, 0) -- (4, 1) -- (4.5, 1) -- (4.5, 0) -- (5, 0) -- (5, 1) -- (5.5, 1) -- (5.5, 0) -- (8.5, 0);
\draw (0, 0) node[left] {$\ell$};
\draw (3, .5) node[left] {$\beta$};
\draw (3.25, 1) node[above] {$\alpha$};
\draw[->] (3, 0) -- (5, -2);
\draw[->] (3, 1) -- (1, 3);
\draw[->] (3.5, 1) -- (5.5, 3);
\draw[->] (3.5, 0) -- (1.5, -2);
\draw[->] (4, 0) -- (6, -2);
\draw[->] (4, 1) -- (2, 3);
\draw[->] (4.5, 1) -- (6.5, 3);
\draw[->] (4.5, 0) -- (2.5, -2);
\draw[->] (5, 0) -- (7, -2);
\draw[->] (5, 1) -- (3, 3);
\draw[->] (5.5, 1) -- (7.5, 3);
\draw[->] (5.5, 0) -- (3.5, -2);
\end{tikzpicture}
\caption{A cr\'eneau on the segment $\ell$.}
\end{figure}
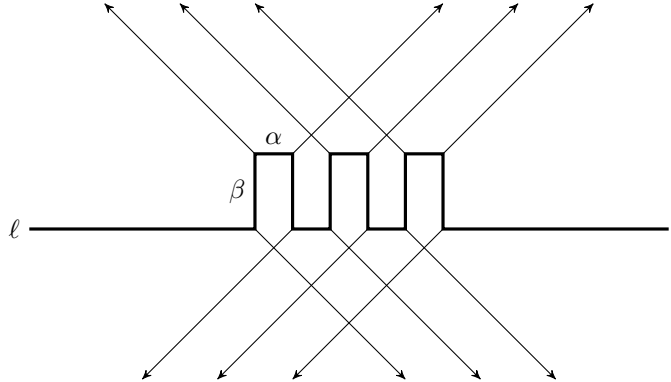
\end{definition}

We prove that given a segment with rational slope in $\mathbb{R}^{2}$ of length $x$, we can can modify its length to be any $\varepsilon>x$, by tropically modifying $\Gamma$ and adding a suitable cr\'eneau as the embedded edge of the tropical modification of $\Gamma$.\\

\subsection{Elongation of the Tropical Length of a Segment}

:\\
\begin{lemma}
Let $x, \alpha, \varepsilon\in \mathbb{R}^{+}$ with $\alpha>x$. Let $\ell$ be the segment joining the points $\left(0, 0\right)$ and $\left(x, 0\right)$ in $\mathbb{R}^{2}$. We can insert suitable cr\'eneau on $\ell$ so that the tropical length of the resultant path is $\alpha$ and the whole path is contained in the rectangle $\left[0, x\right]\times \left[-\varepsilon, \varepsilon\right]\subset \mathbb{R}^{2}$.\\
\end{lemma}

\begin{proof}
The initial tropical length of this segment is $x$. If $\alpha>x$ then we insert a cr\'eneau with endpoints on the endpoints $\left(\frac{x}{3}, 0\right)$ and $\left(\frac{2x}{3}, 0\right)$. The tropical length of the segments which form the cr\'eneau is equal to their euclidean length, since the primitive vectors in the directions of these segments are $\left(1, 0\right)^{t}$ and $\left(0, 1\right)^{t}$.\\
\\
Now, the total length of the horizontal segments in the cr\'eneau is $\frac{x}{3}$. If the length of a vertical segment forming the cr\'eneau is $\delta$, and there are $m$ such segments, then the total length of those segments is $m\delta$. Hence, we need that
\\
\begin{align*}
\alpha-\frac{2x}{3}&=\frac{x}{3}+m\delta \\
\frac{\alpha-x}{m}&=\delta\leq\varepsilon
\end{align*}
\\
so, letting $m>\frac{\alpha-x}{\varepsilon}$, we can choose our appropriate $\delta$ and hence the length of our new path is $\alpha$.
\end{proof}

Since rotation by $\theta$ angle is equivalent to the action of
\\
\[M_{\theta}=
\begin{pmatrix}
\cos\left(\theta\right) & -\sin\left(\theta\right)\\
\sin\left(\theta\right) & \cos\left(\theta\right) \\
\end{pmatrix}
\in SL_{2}\left(\mathbb{R}\right)\]
\\
over $\mathbb{R}^{2}$, letting $\theta\in\left[0, \frac{\pi}{2}\right)$ such that $\tan\left(\theta\right)=\frac{p}{q}$ where $p, q$ are non negative relatively prime integers, we have that for an integral vector $\left(x, y\right)^{t}\in \mathbb{Z}^{2}$,
\\
\[M_{\theta}
\begin{pmatrix}
x \\
y
\end{pmatrix}=
\frac{1}{\sqrt{p^{2}+q^{2}}}\begin{pmatrix}
p& -q \\
q & p
\end{pmatrix}
\begin{pmatrix}
x \\
y
\end{pmatrix}
=
\frac{1}{\sqrt{p^{2}+q^{2}}}
\begin{pmatrix}
px-qy \\
qx+py
\end{pmatrix}.\]
\\
Then it follows that
\\
\[M_{\theta}
\begin{pmatrix}
1 \\
0
\end{pmatrix}=
\frac{1}{\sqrt{p^{2}+q^{2}}}
\begin{pmatrix}
p \\
q
\end{pmatrix}
\quad\quad\quad
M_{\theta}
\begin{pmatrix}
0 \\
1
\end{pmatrix}=
\frac{1}{\sqrt{p^{2}+q^{2}}}
\begin{pmatrix}
-q \\
p
\end{pmatrix}
\]
\\
So after the action of $M_{\theta}$, the primitive vectors parallel to the image of $\left(1, 0\right)^{t}$ and $\left(0, 1\right)^{t}$ are $\left(p, q\right)^{t}$ and $\left(-q, p\right)^{t}$, respectively. Then the tropical length of both segments is affected by a factor of $\sqrt{p^{2}+q^{2}}$ by this rotation. If $\theta=\frac{\pi}{2}$, we have that
\\
\[M_{\frac{\pi}{2}}
\begin{pmatrix}
1 \\
0
\end{pmatrix}
=\begin{pmatrix}
0 \\
1
\end{pmatrix}
\quad\quad\quad
M_{\frac{\pi}{2}}
\begin{pmatrix}
0 \\
1
\end{pmatrix}
=\begin{pmatrix}
-1 \\
0
\end{pmatrix}
\]
\\
So in this case the tropical length is preserved. Now, reflection by the $x-$axis or $y-$axis just changes the sign of one coordinate of our vectors, and therefore, tropical lengths are preserved by reflections along these axis. This extends the value that $\theta$ can take to
\\
\[\Theta=\left\{\frac{\left(2k+1\right)\pi}{2}\mid k\in \mathbb{Z}\right\}\cup\left\{x\mid \tan\left(x\right)\in \mathbb{Q}\right\}\]
\\
and our proof of our previous lemma is still valid. Namely,

\begin{lemma}
Let $x, \alpha, \varepsilon\in \mathbb{R}^{+}$ with $\alpha>x$. Let $\ell$ be the segment joining the points $\left(0, 0\right)$ and $\left(x\cos\left(\theta\right), x\sin\left(\theta\right)\right)$ in the plane, with $\theta\in \Theta$. We can insert suitable cr\'eneaux on $\ell$ so that the tropical length of the resultant path is $\alpha$ and the whole path is contained in the rectangle $\left[0, x\right]\times \left[-\varepsilon, \varepsilon\right]\subset \mathbb{R}^{2}$, rotated by an angle of $\theta$ around $\left(0, 0\right)$.\\
\end{lemma}

\subsection{Main Result}
:\\
\\
Now we will prove our main result using the previous lemmas, and discuss an application of it in the next section.

\begin{theorem}
Every abstract metric graph $\Gamma$ admits an isometric balanced embedding up to $\text{cross}\left(\Gamma\right)$ crossings restricted to $\Gamma$.
\end{theorem}

\begin{proof}Let $\Gamma$ be an abstract metric graph. By lemma 1.1, there exists an embedding $G$ of a tropical modification $\Gamma'$  of $\Gamma$ into the plane, such that every edge of $G$ is a straight line with rational slope, and there are exactly $\text{cross}\left(\Gamma\right)$ crossings when we restrict the embedding to $\Gamma$. If an edge in $\Gamma$ had length $l$ and was subdivided into $n$ edges, we define the lengths of the subdivided edges of $\Gamma'$ to be $\frac{l}{n}$.\\
\\
Let $\alpha$ be the smallest edge length of $\Gamma'$. We now scale down $G$ in the plane so that the tropical length corresponding to every segment of $G$ is less than $\alpha$. Let $\beta$ be the minimum euclidean length of the segments in $G$ and let
\\
\[H=G-\bigcup_{v\text{ vertex of }G}{N_{\frac{\beta}{4}}\left(v\right)}\]
\\
be the finite disjoint union of closed segments and rays obtained from $G$ by removing a ball of radius $\frac{\beta}{4}$ for every vertex $v$ of $G$. We can find $\varepsilon >0$ such that all the $\varepsilon-$neighbourhoods of these segments are disjoint.\\
\\
By lemma 1.3, we can insert cr\'eneau on the edges of $H$ which are contained in their respective edge's $\varepsilon-$neighborhood, in order to make their tropical length match the length of the corresponding edge of $\Gamma'$. Since the cr\'eneaux added to $H$ are contained in disjoint neighborhoods, no new interseccions are created in $G$.\\
\\
Moreover, since every segment has a rational slope, for each vertex $p\in G$, the sum of the primitive vectors of the edges and rays going out from $p$ must be a vector with integer coordinates. Let this vector be $\mathbf{v}_{p}=\left(x_{p}, y_{p}\right)^{t}$ and let $\omega_{p}=\gcd\left(x_{p}, y_{p}\right)$. Then adding an infinite ray $\vec{r}_{p}$ in the direction of $-\mathbf{v}_{p}$ with edge weight equal to $\omega_{p}$, along with weights equal to $1$ on the finite edges, makes the balancing condition hold on $p$. Doing this for every vertex $p\in G$, we produce a balanced graph.
\end{proof}

\section{Application to Algebraic Geometry}

In this section we present an application of our construction along with a result of Baker and Rabinoff. We start by recalling some of the assumptions needed, and some results about non-archimedean fields.\\

\subsection{$\Lambda$-rational Tropical Meromorphic Functions}

:\\
\\
Let $K$ be an algebraically closed field which is complete with respect to a nontrivial, non-Archimedean valuation $\text{val}:K\to \mathbb{R}\cup\left\{\infty\right\}$ and let $\Lambda=\text{val}\left(K^{\times}\right)$ be its value group. We begin this section by proving some results about $\Lambda$.\\

\begin{lemma}
If $\lambda\in \Lambda$ and $q\in \mathbb{Q}$ then $q\lambda\in \Lambda$.
\end{lemma}

\begin{proof}
Let $q=\frac{n}{m}>0$ with $m, n\in \mathbb{Z}^{+}$. Let $k\in K$ such that $\text{val}\left(k\right)=\lambda$. Then we have that
\\
\[\text{val}\left(k^{n}\right)=n\cdot\text{val}\left(k\right)=n\lambda.\]
\\
Since $K$ is algebraically closed, the equation $x^{m}-k^{n}=0$ has a solution in $K$. Let $z$ be such solution. Then we have that
\\
\[m\cdot \text{val}\left(z\right)=\text{val}\left(z^{m}\right)=\text{val}\left(k^{n}\right)=n\lambda\]
\\
and therefore $\text{val}\left(z\right)=q\lambda$, so $q\lambda\in \Lambda$. Note that $\text{val}\left(1\right)=\text{val}\left(1\right)+\text{val}\left(1\right)$ so $\text{val}\left(1\right)=0$. Hence
\\
\[\text{val}\left(z\right)+\text{val}\left(z^{-1}\right)=\text{val}\left(1\right)=0\]
\\
so $\text{val}\left(z^{-1}\right)=-q\lambda$. Therefore, for every $q\in \mathbb{Q}$ we have $q\lambda\in \Lambda$.\\
\end{proof}

\begin{lemma}
Let $\left(x_{1}, y_{1}\right), \left(x_{2}, y_{2}\right)\in \Lambda^{2}$ such that the segment in $\mathbb{R}^{2}$ which joins these points has rational slope. Then the tropical length $\lambda$ of this segment lies in $ \Lambda$.
\end{lemma}

\begin{proof}
Let $\left(m, n\right)^{t}$ be the primitive vector of this segment, considered as a ray going from $\left(x_{1}, y_{1}\right)$. Then, we have that
\\
\[\lambda
\begin{pmatrix}
m \\
n
\end{pmatrix}=
\begin{pmatrix}
x_{2}-x_{1} \\
y_{2}-y_{1}
\end{pmatrix}.\]
\\
One of $m, n$ must be non-zero, so without loss of generality, assume $m\neq 0$. Then $\lambda=\frac{x_{2}-x_{1}}{m}$ and since $x_{1}, x_{2}\in \Lambda$, we have that $x_{2}-x_{1}\in \Lambda$. By lemma 1, we have that $\frac{x_{2}-x_{1}}{m}\in \Lambda$ and therefore $\lambda\in \Lambda$.
\end{proof}

\begin{lemma}
Let $\left(x, y\right)\in \Lambda^{2}$ and let $\lambda\in \Lambda$ be a positive real number. If a segment in $\mathbb{R}^{2}$ with rational slope and endpoints in $\left(x, y\right)$ and $\left(x_{1}, y_{1}\right)$ has tropical length equal to $\lambda$, then $\left(x_{1}, y_{1}\right)\in \Lambda^{2}$.\\
\end{lemma}

\begin{proof}
Let $\left(m, n\right)^{t}$ be the primitive vector of our segment, considered as a ray going from $\left(x, y\right)$. Then, we have that
\\
\[\lambda
\begin{pmatrix}
m \\
n
\end{pmatrix}=
\begin{pmatrix}
x_{1}-x \\
y_{1}-y
\end{pmatrix}.\]
\\
Hence, we have that $x_{1}=x+m\lambda$ and $y_{1}=y+n\lambda$. Since $\lambda\in \Lambda$ then $n\lambda, m\lambda\in \Lambda$, and therefore $x_{1}, y_{1}\in \Lambda$ and $\left(x_{1}, y_{1}\right)\in \Lambda^{2}$.\\
\end{proof}

The following result is immediate from the previous lemma.\\

\begin{corollary}
Let $\left(x_{1}, y_{1}\right), \left(x_{2}, y_{2}\right)\in \Lambda^{2}$ be such that the segment in $\mathbb{R}^{2}$ joining them has rational slope. Then a point $\left(X, Y\right)$ on that segment is in $\Lambda^{2}$
if and only if its tropical distance to one of the segment endpoints lies in $\Lambda$.\\
\end{corollary}

Now, we study the skeletons of $K$-curves. Consider a connected, projective $K$-curve $X$, and consider its analytification $X^{\text{an}}$ in the sense of Berkovich \cite{Berkovich90}. The incidence graph $\Gamma$ of the special fiber of a semistable $R-$model $\mathfrak{X}$ of $X$ has the structure of a metric graph with edge lengths in $\Lambda$ \cite[p. 1]{BakerRabinoff13}. This graph $\Gamma$ is called the {\it skeleton} of $X$ and has an inclusion $\Gamma\hookrightarrow X^{\text{an}}$ by a deformation retract $\tau:X^{\text{an}}\to \Gamma$. Let $\Gamma\left(\Lambda\right)$ be the set of all points whose distance to every vertex in $\Gamma$ belongs to $\Lambda$.\\
  
\begin{definition}
A continuous function $f:\Gamma\to\mathbb{R}$ is {\bf $\Lambda-$rational} if $f\left(\Gamma\left(\Lambda\right)\right)\subset \Lambda$ and all points at which $f$ is not differentiable are contained in $\Gamma\left(\Lambda\right)$.\\
\end{definition}

\begin{definition}
A {\bf tropical meromorphic function} of $\Gamma$ is a continuous function $f:\Gamma\to \mathbb{R}$ which is piecewise affine and has integer slopes.\\ 
\end{definition}

Given the skeleton of a $K$-curve $X$, our goal will be to find two $\Lambda$-rational tropical meromorphic functions $f, g:\Gamma\to \mathbb{R}$ via our construction in lemma 1.1, in order to use the following result by Baker and Rabinoff.\\

\begin{proposition}
{\bf (Baker, Rabinoff)} \cite[p. 2, Corollary 1.2]{BakerRabinoff13} Let $K$ be a complete algebraically closed non archimedean field and let $X$ be a $K-$curve. Let $\Gamma$ be a skeleton of $X$ and $F:\Gamma\to \mathbb{R}$ be a continuous function. There exists a non zero rational function $f\in K\left(X\right)$ such that $F=-\log\left|f\right|\mid_{\Gamma}$ if and only if $F$ is a $\Lambda-$rational tropical meromorphic function.\\
\end{proposition}

Namely, we prove the following two lemmas.\\

\begin{lemma}
Let $\Gamma$ be a skeleton of a $K$-curve. There exist an isometric balanced embedding of $\Gamma$ with $\text{cross}\left(\Gamma\right)$ crossings such that all the vertex of the embedded graph are in $\Lambda^{2}$.\\
\end{lemma}

\begin{proof}
Reconsider our construction in lemma 1.1. Let $G$ a rational polyhedral complex embedding of a modification which vertices are mapped to $\mathbb{Q}^{2}$. Let $\lambda \in \Lambda$ such that $\lambda \in (0,1)$. This $\lambda$ exists because $\Lambda$ is dense in the $\mathbb{R}$. Rescaling $G$ by $\lambda$, all its vertices lie in $\Lambda^{2}$, and the tropical lengths of the edges lie in $\Lambda$ by corollary 3.1, since the slopes remain rational after rescaling.\\
\\
We proceed to add the cr\'eneaux as in lemma 1.3, in such a way that the vertices of each cr\'eneau are in $\Lambda^{2}$ and the lengths of the edges of $\Gamma$ (which are in $\Lambda$) are equal to the tropical length of the corresponding edges in $G$. We insert the cr\'eneaux the following way. Suppose that the length of an edge in $\Gamma$ is $l\in \Lambda$ and we have a segment which is part of the embedding of that edge, which has tropical length $x$. Since the vertices of this segment are in $\Lambda^{2}$ then we have that $x\in \Lambda$. Hence, we have that $\frac{x}{n}, \frac{l-x}{n}\in \Lambda$ for any positive integer $n$. Consider then the following cr\'eneau on this segment.
\begin{figure}[h]
\begin{tikzpicture}
\centering
\draw[very thick] (0, 0) -- (3, 0) -- (3, 1) -- (3.5, 1) -- (3.5, 0) -- (4, 0) -- (4, 1) -- (4.5, 1) -- (4.5, 0) -- (5, 0) -- (5, 1) -- (5.5, 1) -- (5.5, 0) -- (8.5, 0);
\fill (0, 0) circle (2pt) node[left] {$\left(p_{1}\lambda, q_{1}\lambda\right)$};
\fill (8.5, 0) circle (2pt) node[right] {$\left(p_{2}\lambda, q_{2}\lambda\right)$};
\draw (3, .5) node[left] {$\frac{l-x}{2n}$};
\draw (3.25, 1) node[above] {$\frac{x}{6n-3}$};
\draw (1.5, 0) node[below] {$\frac{x}{3}$};
\draw (7, 0) node[below] {$\frac{x}{3}$};
\draw[->] (3, 0) -- (5, -2);
\draw[->] (3, 1) -- (1, 3);
\draw[->] (3.5, 1) -- (5.5, 3);
\draw[->] (3.5, 0) -- (1.5, -2);
\draw[->] (4, 0) -- (6, -2);
\draw[->] (4, 1) -- (2, 3);
\draw[->] (4.5, 1) -- (6.5, 3);
\draw[->] (4.5, 0) -- (2.5, -2);
\draw[->] (5, 0) -- (7, -2);
\draw[->] (5, 1) -- (3, 3);
\draw[->] (5.5, 1) -- (7.5, 3);
\draw[->] (5.5, 0) -- (3.5, -2);
\end{tikzpicture}
\caption{Cr\'eneau with lengths in $\Lambda$.}
\end{figure}
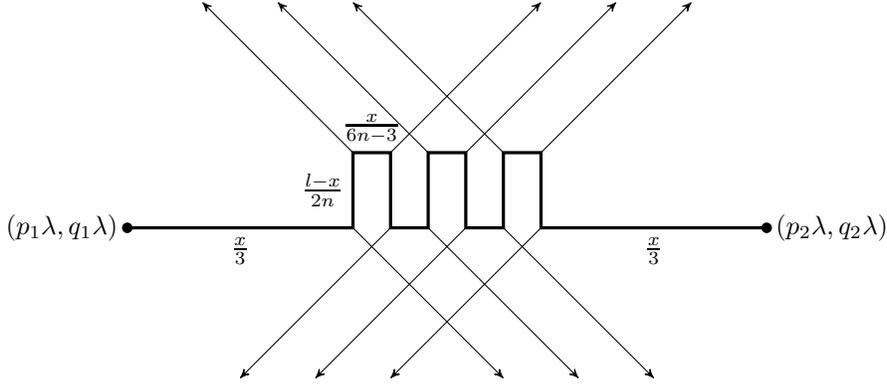
The length of our new path is the sum of the lengths of each segment forming it, which is
\\
\[\frac{x}{3}+\frac{x}{3}+\left(2n-1\right)\left(\frac{x}{6n-3}\right)+2n\left(\frac{l-x}{2n}\right)=x+l-x=l.\]
\\
Since the vertices of our initial segment are in $\Lambda^{2}$ and each of the new segments conforming the new path have lengths in $\Lambda$, then each one of the new vertices must be in $\Lambda^{2}$.\\
\\
After inserting suitable infinite rays in order for the balancing condition to hold on every vertex, we obtain a balanced graph which is isometric to $\Gamma$ and has exactly $\text{cross}\left(\Gamma\right)$ crossings.\\
\end{proof}

\begin{lemma}
Let $\Gamma$ be a skeleton of a $K$-curve and let $\phi:\Gamma\to \mathbb{R}^{2}$ be the mapping of $\Gamma$ to the balanced graph $G$ in lemma 3.4. Then the functions $f, g:\Gamma \to \mathbb{R}$ defined by
\\
\begin{align*}
f\left(p\right)&=\pi_{x}\left(\phi\left(p\right)\right) \\
g\left(p\right)&=\pi_{y}\left(\phi\left(p\right)\right) \\
\end{align*}
for every $p\in \Gamma$, are $\Lambda$-rational tropical meromorphic functions.\\
\end{lemma}

\begin{proof}
First note that $f$ and $g$ are linear on each edge of $G$. This implies that, if the primitive vector parallel to a segment $l$ with tropical length $\lambda$ with vertices $p$ and $q$ is $\left(m, n\right)^{t}$, then
\\
\[\lambda
\begin{pmatrix}
m \\
n
\end{pmatrix}=
\phi\left(q\right)-\phi\left(p\right)\]
\\
Hence, the slopes of $f$ on $l$ is
\\
\[\frac{\pi_{x}\left(\phi\left(q\right)\right)-\pi_{x}\left(\phi\left(p\right)\right)}{\lambda}=m\]
\\
and the slope of $g$ on $l$ is
\\
\[\frac{\pi_{y}\left(\phi\left(q\right)\right)-\pi_{y}\left(\phi\left(p\right)\right)}{\lambda}=n\]
\\
It is clear that $m$ and $n$ are integers and $\gcd\left(m, n\right)=1$. This shows that $f, g$ are tropical meromorphic functions. It remains to show that $f$ and $g$ are $\Lambda-$rational. Let $\Gamma\left(\Lambda\right)$ be set of points which have distance in $\Lambda$ to all vertices of $\Gamma$. By corollary 3.1, every point $p\in \Gamma\left(\Lambda\right)$ satisfies that $p\in \Lambda^{2}$. Hence, we must have that
\\
\[f\left(\Gamma\left(\Lambda\right)\right), g\left(\Gamma\left(\Lambda\right)\right)\subset \Lambda.\]
\\
Therefore, $f, g$ are $\Lambda-$rational tropical meromorphic functions.\\
\end{proof}

\subsection{Main Application}

:\\
\\
We conclude by proving our main application to Baker and Rabinoff's result.

\begin{theorem}
Let $X$ be a $K-curve$ and let $\Gamma$ be a skeleton of $X$. There exists a rational map $f:X\dashrightarrow \mathbb{P}^{2}$ such that the restriction $\text{trop}\left(f\right)=-\log\left|f\right|$ restricted to $\Gamma$ is an isometry up to $\text{cross}\left(\Gamma\right)$ crossings.\\
\end{theorem}

\begin{proof}
Consider our functions from corollary 3.5. By proposition 3.1, there exist rational functions $F, G\in K\left(X\right)$ such that $f=-\log\left|F\right|\mid_{\Gamma}$ and $g=-\log\left|G\right|\mid_{\Gamma}$, since $f$ and $g$ are $\Lambda-$rational tropical meromorphic. Therefore, the map $\left(F, G\right):X\to \mathbb{P}^{2}$ satisfies that the restriction to $\Gamma$ of its tropicalization $\left(f, g\right):\Gamma\to \mathbb{R}^{2}$ is an isometry (since for every pair of slopes $m$ and $n$ of $f$ and $g$, respectively, on the same segment, we showed $\gcd\left(m, n\right)=1$) with respect to the tropical length, with the exception of $\text{cross}\left(\Gamma\right)$ crossings.
\end{proof}

\section{Acknowledgements}

{\small I would like to thank my mentor Sylvain Carpentier for his invaluable help and guidance throughout this project, in learning and developing new mathematics. I would like to thank Dr. Melody Chan for suggesting this project, for her time, and for the very insightful discussions we had about this project. I would like to thank prof. David Jerison and prof. Ankur Moitra for offering very insightful input and comments which greatly improved the overall quality of the research. Finally, I would like to thank Dr. Slava Gerovitch and MIT for organizing the SPUR Program, thanks to which doing this research paper was possible.}

\end{document}